\documentclass[12pt]{article}
\usepackage[utf8]{inputenc}
\usepackage[english]{babel}
\usepackage{amsthm}
\usepackage{amsmath}
\usepackage{amsfonts}
\usepackage{amssymb}
\usepackage{algorithm}
\usepackage{algorithmicx}
\usepackage{algpseudocode}

\usepackage{amsopn}
\usepackage{graphicx,float}
\usepackage{enumitem}
\usepackage{multirow}
\usepackage[toc,page]{appendix}

\usepackage{pdfpages}
\usepackage[nottoc,numbib]{tocbibind}

\usepackage{dirtytalk}

\usepackage[pagebackref=true,colorlinks=true]{hyperref}
\hypersetup{
pdffitwindow=true,
linkcolor=blue,
citecolor=green,
urlcolor=cyan}
\usepackage{babelbib}
\setbtxfallbacklanguage{english}

\usepackage{geometry}
\usepackage{blindtext}
\usepackage{fancyhdr}

\theoremstyle{plain}
\newtheorem{thm}{\protect\theoremname}
\theoremstyle{remark}
\newtheorem{claim}{\protect\claimname}
\newenvironment{lyxlist}[1]
{\begin{list}{}
		{\settowidth{\labelwidth}{#1}
			\setlength{\leftmargin}{\labelwidth}
			\addtolength{\leftmargin}{\labelsep}
			}}
	{\end{list}}
\theoremstyle{plain}
\newtheorem{lem}{\protect\lemmaname}
\theoremstyle{plain}
\newtheorem{cor}{\protect\corollaryname}

\providecommand{\claimname}{Claim}
\providecommand{\corollaryname}{Corollary}
\providecommand{\lemmaname}{Lemma}
\providecommand{\theoremname}{Theorem}

\providecommand{\keywords}[1]
{
	\small	
	\textbf{\textit{Keywords---}} #1
}

\theoremstyle{definition}

\theoremstyle{definition}

\theoremstyle{definition}
\theoremstyle{definition}

\theoremstyle{definition}

\let\inf\relax 
\DeclareMathOperator*\inf{\vphantom{p}inf}

\usepackage{mathtools}

\title{Approximation by Finite Mixtures of Continuous Density Functions That Vanish at Infinity}
\author{T. Tin~Nguyen,~
	Hien D.~Nguyen$^\dagger$,~
	Faicel~Chamroukhi,\\
	~and~Geoffrey J.~McLachlan\thanks{T. Tin Nguyen and Faicel Chamroukhi are with department of Mathematics and Computer Science, Normandie University, UNICAEN, UMR CNRS LMNO, Caen, France. Hien D. Nguyen is with Department of Mathematics and Statistics, La Trobe University,
		Melbourne, Victoria, Australia. Geoffrey J.~McLachlan is with School of Mathematics and Physics, University of Queensland, St. Lucia, Brisbane, Australia. E-mail: trung-tin.nguyen@unicaen.fr, ~h.nguyen5@latrobe.edu.au,
		~faicel.chamroukhi@unicaen.fr, ~g.mclachlan@uq.edu.au. $^\dagger$Corresponding author.}% <-this % stops a space
		\vspace{-8ex}
  \date{}
	}

\begin{document}
\maketitle
%\tableofcontents
\begin{abstract}
	Given sufficiently many components, it is often cited that  finite
	mixture models can approximate any other probability density function
	(pdf) to an arbitrary degree of accuracy. Unfortunately, the nature
	of this approximation result is often left unclear. We prove that
	finite mixture models constructed from pdfs in $\mathcal{C}_{0}$
	can be used to conduct approximation of various classes of approximands
	in a number of different modes. That is, we prove approximands in
	$\mathcal{C}_{0}$ can be uniformly approximated, approximands in
	$\mathcal{C}_{b}$ can be uniformly approximated on compact sets,
	and approximands in $\mathcal{L}_{p}$ can be approximated with respect to the
	$\mathcal{L}_{p}$, for $p\in\left[1,\infty\right)$. Furthermore,
	we also prove that measurable functions can be approximated, almost
	everywhere.
\end{abstract}


\section{Introduction}
Let $x$ be an element in the Euclidean space, defined by
$\mathbb{R}^{n}$ and the norm $\left\Vert \cdot\right\Vert _{2}$,
for some $n\in\mathbb{N}$. Let $f:\mathbb{R}^{n}\rightarrow\mathbb{R}$ be a function,
such that $f\ge0$, everywhere, and $\int f\text{d}\lambda=1$, where
$\lambda$ is the Lesbegue measure. We say that $f$ is a probability
density function (pdf) on the domain $\mathbb{R}^{n}$ (an expression
that we will drop, from hereon in). Let $g:\mathbb{R}^{n}\rightarrow\mathbb{R}$
be another pdf, and for each $m\in\mathbb{N}$, define the functional
class:
\begin{eqnarray*}
	\mathcal{M}_{m}^{g}&=&\left\{ h:h\left(x\right)=\sum_{i=1}^{m}c_{i}\frac{1}{\sigma_{i}^{n}}g\left(\frac{x-\mu_{i}}{\sigma_{i}}\right), \mu_{i}\in\mathbb{R}^{n}\text{, }\sigma_{i}\in\mathbb{R}_{+}\text{, }c\in\mathbb{S}^{m-1}\text{, }i\in\left[m\right]\right\} \text{,}
\end{eqnarray*}
where $c^\top=(c_1,\dots,c_m)$, $\mathbb{R}_{+}=\left(0,\infty\right)$,
\[
\mathbb{S}^{m-1}=\left\{ c\in\mathbb{R}^{m}:\sum_{i=1}^{m}c_{i}=1\text{ and }c_{i}\ge0,\forall i\in\left[m\right]\right\} \text{,}
\]
$\left[m\right]=\left\{ 1,\dots,m\right\} $, and $\left(\cdot\right)^{\top}$
is the matrix transposition operator. We say that any $h\in\mathcal{M}_{m}^{g}$
is a $m\text{-component}$ location-scale finite mixture
of the pdf $g$.

The study of pdfs in the class $\mathcal{M}_{m}^{g}$ is an evergreen
area of applied and technical research, in statistics. We point the interested reader to the many comprehensive books on the topic, such as \cite{Everitt1981},\cite{Titterington1985}, \cite{McLachlan1988}, \cite{Lindsay1995},
\cite{McLachlan2000}, \cite{Fruwirth-Schatter2006}, \cite{Schlattmann:2009aa},
\cite{mengersen2011mixtures},
and \cite{FruwirthShnatter:2019aa}.

Much of the popularity of finite mixture models stem from the folk
theorem, which states that for any density $f$, there exists an $h\in\mathcal{M}_{m}^{g}$,
for some sufficiently large number of components $m\in\mathbb{N}$,
such that $h$ approximates $f$ arbitrarily closely, in some sense.
Examples of this folk theorem come in statements such as: ``provided
the number of component densities is not bounded above, certain forms
of mixture can be used to provide arbitrarily close approximation
to a given probability distribution'' \cite[p. 50]{Titterington1985},
``the {[}mixture{]} model forms can fit any distribution and significantly
increase model fit'' \cite[p. 173]{WalkerBen-Akiva2011}, and ``a
mixture model can approximate almost any distribution'' \cite[p. 500]{Yona2011}.
Other statements conveying the same sentiment are reported in \cite{Nguyen:2019aa}.
There is a sense of vagary in the reported statements, and little
is ever made clear regarding the technical nature of the folk theorem.

In order to proceed, we require the following definitions. We say that $f$ is compactly supported on $\mathbb{K}\subset\mathbb{R}^{n}$,
if $\mathbb{K}$ is compact and if $\mathbf{1}_{\mathbb{K}^{\complement}}f=0$,
where $\mathbf{1}_{\mathbb{X}}$ is the indicator function that takes
value 1 when $x\in\mathbb{X}$ and $0$, elsewhere, and $\left(\cdot\right)^{\complement}$
is the set complement operator (i.e., $\mathbb{X}^{\complement}=\mathbb{R}^{n}\backslash\mathbb{X}$).
Here, $\mathbb{X}$ is a generic subset of $\mathbb{R}^{n}$. Furthermore,
we say that $f\in\mathcal{L}_{p}\left(\mathbb{X}\right)$ for any
$1\le p<\infty$, if
\[
\left\Vert f\right\Vert _{\mathcal{L}_{p}\left(\mathbb{X}\right)}=\left(\int\left|\mathbf{1}_{\mathbb{X}}f\right|^{p}\text{d}\lambda\right)^{1/p}<\infty\text{,}
\]
and for $p=\infty$, if
\begin{align*}
&\left\Vert f\right\Vert _{\mathcal{L}_{\infty}\left(\mathbb{X}\right)}=\inf\left\{a \ge 0: \lambda \left( \left\{x \in \mathbb{X}: \left|f(x)\right| > a\right\} \right) = 0 \right\}<\infty, 
\end{align*}
where we call $\left\Vert \cdot\right\Vert _{\mathcal{L}_{p}\left(\mathbb{X}\right)}$
the $\mathcal{L}_{p}\text{-norm}$ on $\mathbb{X}$. When $\mathbb{X}=\mathbb{R}^{n}$,
we shall write $\left\Vert \cdot\right\Vert _{\mathcal{L}_{p}\left(\mathbb{R}^{n}\right)}=\left\Vert \cdot\right\Vert _{\mathcal{L}_{p}}$.
In addition, we define the so-called Kullback-Leibler divergence, see \cite{KullbackLeibler1951}, 
between any two pdfs $f$ and $g$ on $\mathbb{X}$ as
\[
\text{KL}_{\mathbb{X}}\left(f,g\right)=\int\mathbf{1}_{\mathbb{X}}f\log\left(\frac{f}{g}\right)\text{d}\lambda\text{.}
\]

In \cite{Nguyen:2019aa}, the approximation of pdfs $f$ by the class
$\mathcal{M}_{m}^{g}$ was explored in a restrictive setting. Let
$\left\{ h_{m}^{g}\right\} $ be a sequence of functions that draw
elements from the nested sequence of sets $\left\{ \mathcal{M}_{m}^{g}\right\} $
(i.e., $h_{1}^{g}\in\mathcal{M}_{1}^{g},h_{2}^{g}\in\mathcal{M}_{2}^{g},\dots$).
The following result of \cite{ZeeviMeir1997} was presented in \cite{Nguyen:2019aa},
along with a collection of its implications, such as the results of
from \cite{LiBarron1999} and \cite{RakhlinPanchenkoMukherjee2005}.
\begin{thm}[Zeevi and Meir, 1997]\label{thm zeevi}If $$f\in\left\{ f:\mathbf{1}_{\mathbb{K}}f\ge\beta\text{, }\beta>0\right\} \cap\mathcal{L}_{2}\left(\mathbb{K}\right)$$
	and $g$ are pdfs and $\mathbb{K}$ is compact, then there exists
	a sequence $\left\{ h_{m}^{g}\right\} $ such that
	\[
	\lim_{m\rightarrow\infty}\left\Vert f-h_{m}^{g}\right\Vert _{\mathcal{L}_{2}\left(\mathbb{K}\right)}=0\text{ and }\lim_{m\rightarrow\infty}\text{KL}_{\mathbb{K}}\left(f,h_{m}^{g}\right)=0\text{.}
	\]
\end{thm}
Although powerful, this result is restrictive in the sense that it
only permits approximation in the $\mathcal{L}_{2}$ norm on compact
sets $\mathbb{K}$, and that the result only allows for approximation
of functions $f$ that are strictly positive on $\mathbb{K}$. In
general, other modes of approximation are desirable, in particular
approximation in $\mathcal{L}_{p}\text{-norm}$ for $p=1$ or $p=\infty$
are of interest, where the latter case is generally referred to as
uniform approximation. Furthermore, the strict-positivity assumption,
and the restriction on compact sets limits the scope of applicability
of Theorem \ref{thm zeevi}. An example of an interesting application
of extensions beyond Theorem \ref{thm zeevi} is within the $\mathcal{L}_{1}\text{-norm}$
approximation framework of \cite{Devroye:2000aa}.

Let $g:\mathbb{R}^{n}\rightarrow\mathbb{R}$ again be a pdf. Then,
for each $m\in\mathbb{N}$, we define
\begin{align*}
\mathcal{N}_{m}^{g}&=\left\{ h:h\left(x\right)=\sum_{i=1}^{m}c_{i}\frac{1}{\sigma_{i}^{n}}g\left(\frac{x-\mu_{i}}{\sigma_{i}}\right)\text{, }\mu_{i}\in\mathbb{R}^{n}\text{, }\sigma_{i}\in\mathbb{R}_{+}\text{, }c_i\in\mathbb{R}\text{, }i\in\left[m\right]\right\} \text{,}
\end{align*}

which we call the set of $m\text{-component}$ location-scale linear
combinations of the pdf $g$. In the past, results regarding approximations
of pdfs $f$ via functions $\eta\in\mathcal{N}_{m}^{g}$ have been
more forthcoming. For example, in the case of $g=\phi$, where
\begin{equation}
\phi\left(x\right)=\left(2\pi\right)^{-n/2}\exp\left(-\left\Vert x\right\Vert _{2}^{2}/2\right)\text{,}\label{eq: gaussian pdf}
\end{equation}
is the standard normal pdf. Denoting the class of continuous functions
with support on $\mathbb{R}^{n}$ by $\mathcal{C}$. We have the result
that for every pdf $f$, compact set $\mathbb{K}\subset\mathbb{R}^{n}$,
and $\epsilon>0$, there exists an $m\in\mathbb{N}$ and $h\in\mathcal{N}_{m}^{\phi}$,
such that $\left\Vert f-h\right\Vert _{\mathcal{L}_{\infty}\left(\mathbb{K}\right)}<\epsilon$
\cite[Lem. 1]{Sandberg:2001aa}. Furthermore, upon defining the set
of continuous functions that vanish at infinity by
\begin{align*}
\mathcal{C}_{0}&=\left\{ f\in\mathcal{C}:\forall\epsilon>0,\exists\text{ a compact }\mathbb{K}\subset\mathbb{R}^{n}\text{,} \text{ such that}\left\Vert f\right\Vert _{\mathcal{L}_{\infty}\left(\mathbb{K}^{\complement}\right)}<\epsilon
\right\} \text{,}
\end{align*}

we also have the result: for every pdf $f\in\mathcal{C}_{0}$ and
$\epsilon>0$, there exists an $m\in\mathbb{N}$ and $h\in\mathcal{N}_{m}^{\phi}$,
such that $\left\Vert f-h\right\Vert _{\mathcal{L}_{\infty}}<\epsilon$
\cite[Thm. 2]{Sandberg:2001aa}. Both of the results from \cite{Sandberg:2001aa}
are simple implications of the famous Stone-Weierstrass theorem (cf.
\cite{Stone1948} and \cite{De-Branges:1959aa}).

To the best of our knowledge, the strongest available claim that is
made regarding the folk theorem, within a probabilistic or statistical
context, is that of \cite[Thm. 33.2]{DasGupta2008}. Let $\left\{ \eta_{m}^{g}\right\} $
be a sequence of functions that draw elements from the nested sequence
of sets $\left\{ \mathcal{N}_{m}^{g}\right\} $, in the same manner
as $\left\{ h_{m}^{g}\right\} $. We paraphrase the claim without
loss of fidelity, as follows.
\begin{claim}\label{claim dasgupta}If $f,g\in\mathcal{C}$ are pdfs and $\mathbb{K}\subset\mathbb{R}^{n}$
	is compact, then there exists a sequence $\left\{ \eta_{m}^{g}\right\} $,
	such that
	\[
	\lim_{m\rightarrow\infty}\left\Vert f-\eta_{m}^{g}\right\Vert _{\mathcal{L}_{\infty}\left(\mathbb{K}\right)}=0\text{.}
	\]
\end{claim}
Unfortunately, the proof of Claim \ref{claim dasgupta} is not provided
within \cite{DasGupta2008}. The only reference of the result is
to an undisclosed location in \cite{CheneyLight2000}, which, upon
investigation, can be inferred to be Theorem 5 of \cite[Ch. 20]{CheneyLight2000}.
It is further notable that there is no proof provided for the theorem.
Instead, it is stated that the proof is similar to that of Theorem
1 in \cite[Ch. 24]{CheneyLight2000}, which is a reproduction of
the proof for \cite[Lem. 3.1]{Xu:1993aa}.

There is a major problem in applying the proof technique of \cite[Lem. 3.1]{Xu:1993aa}
in order to prove Claim \ref{claim dasgupta}. The proof of \cite[Lem. 3.1]{Xu:1993aa}
critically depends upon the statement that ``there is no loss of
generality in assuming that $f\left(x\right)=0$ for $x\in\mathbb{R}^{n}\backslash2\mathbb{K}$''.
Here, for $a\in\mathbb{R}_{+}$, $a\mathbb{K}=\left\{ x\in\mathbb{R}^{n}:x=ay\text{, }y\in\mathbb{K}\right\} $.
The assumption is necessary in order to write any convolution with
$f$ and an arbitrary continuous function as an integral over a compact
domain, and then to use a Riemann sum to approximate such an integral.
Subsequently, such a proof technique does not work outside the class
of continuous functions that are compactly supported on $a\mathbb{K}$.
Thus, one cannot verify Claim \ref{claim dasgupta} from the materials
of \cite{Xu:1993aa}, \cite{CheneyLight2000}, and \cite{DasGupta2008},
alone.

Some recent results in the spirit of Claim \ref{claim dasgupta} have
been obtained by \cite{Nestoridis:2007aa} and \cite{Nestoridis:2011aa},
using methods from the study of universal series (see for example in  \cite{Nestoridis:2005aa}).

Let
\[
\mathcal{W}=\left\{ f\in\mathcal{C}_{0}:\sum_{y\in\mathbb{Z}^{n}}\sup_{x\in\left[0,1\right]^{n}}\text{ }\left|f\left(x+y\right)\right|<\infty\right\} 
\]
denote the so-called Wiener's algebra (see, e.g., \cite{Feichtinger:1977aa})
and let
\begin{align*}
\mathcal{V}&=\left\{ f\in\mathcal{C}_{0}:\forall x\in\mathbb{R}^{n}\text{, }\left|f\left(x\right)\right|\le\beta\left(1+\left\Vert x\right\Vert _{2}\right)^{-n-\theta}\text{, } \beta,\theta\in\mathbb{R}_{+}\right\} 
\end{align*}
be a class of functions with tails decaying at a faster rate than
$o\left(\left\Vert x\right\Vert _{2}^{n}\right)$. In \cite{Nestoridis:2011aa},
it is noted that $\mathcal{V}\subset\mathcal{W}$. Further, let
\[
\mathcal{C}_{c}=\left\{ f\in\mathcal{C}:\exists\text{ a compact set }\mathbb{K}\text{, such that }\mathbf{1}_{\mathbb{K}^{\complement}}f=0\right\} \text{,}
\]
denote the set of compactly supported continuous functions. The following
theorem was proved in \cite{Nestoridis:2007aa}.
\begin{thm}
	[Nestoridis and Stefanopoulos, 2007, Thm. 3.2] \label{thm nestoridis2007}If
	$g\in\mathcal{V}$, then the following statements hold.
	\begin{lyxlist}{00.00.0000}
		\item [{(a)}] For any $f\in\mathcal{C}_{c}$, there exists a sequence $\left\{ \eta_{m}^{g}\right\} $ ($\eta_{m}^{g}\in\mathcal{N}_{m}^{g}$),
		such that
		\[
		\lim_{m\rightarrow\infty}\left\Vert f-\eta_{m}^{g}\right\Vert _{\mathcal{L}_{1}}+\left\Vert f-\eta_{m}^{g}\right\Vert _{\mathcal{L}_{\infty}}=0.
		\]
		\item [{(b)}] For any $f\in\mathcal{C}_{0}$, there exists a sequence $\left\{ \eta_{m}^{g}\right\} $ ($\eta_{m}^{g}\in\mathcal{N}_{m}^{g}$),
		such that
		\[
		\lim_{m\rightarrow\infty}\left\Vert f-\eta_{m}^{g}\right\Vert _{\mathcal{L}_{\infty}}=0.
		\]
		\item [{(c)}] For any $1\le p<\infty$ and $f\in\mathcal{L}_{p}$, there
		exists a sequence $\left\{ \eta_{m}^{g}\right\} $ ($\eta_{m}^{g}\in\mathcal{N}_{m}^{g}$), such that
		\[
		\lim_{m\rightarrow\infty}\left\Vert f-\eta_{m}^{g}\right\Vert _{\mathcal{L}_{p}}=0.
		\]
		\item [{(d)}] For any measurable $f$, there exists a sequence $\left\{ \eta_{m}^{g}\right\} $ ($\eta_{m}^{g}\in\mathcal{N}_{m}^{g}$),
		such that
		\[
		\lim_{m\rightarrow\infty}\eta_{m}^{g}=f\text{, almost everywhere.}
		\]
		\item [{(e)}] If $\nu$ is a $\sigma\text{-finite}$ Borel measure on $\mathbb{R}^{n}$,
		then for any $\nu\text{-measurable}$ $f$, there exists a sequence
		$\left\{ \eta_{m}^{g}\right\} $ ($\eta_{m}^{g}\in\mathcal{N}_{m}^{g}$), such that
		\[
		\lim_{m\rightarrow\infty}\eta_{m}^{g}=f,\] almost everywhere, with respect to $\nu$.
	\end{lyxlist}
\end{thm}
The result was then improved upon, in \cite{Nestoridis:2011aa},
whereupon the more general space $\mathcal{W}$ was taken as a replacement
for $\mathcal{V}$, in Theorem \ref{thm nestoridis2007}. Denote the
class of bounded continuous functions by $\mathcal{C}_{b}=\mathcal{C}\cap\mathcal{L}_{\infty}$.
The following theorem was proved in \cite{Nestoridis:2011aa}.
\begin{thm}[Nestoridis et al., 2011, Thm. 3.2]
	\label{thm nestoridis 2011}If
	$g\in\mathcal{W}$, then the following statements are true.
	\begin{lyxlist}{00.00.0000}
		\item [{(a)}] The conclusion of Theorem \ref{thm nestoridis2007}(a) holds,
		with $\mathcal{C}_{c}$ replaced by $\mathcal{C}_{0}\cap\mathcal{L}_{1}$.
		\item [{(b)}] The conclusions of Theorem \ref{thm nestoridis2007}(b)--(e)
		hold.
		\item [{(c)}] For any $f\in\mathcal{C}_{b}$ and compact $\mathbb{K}\subset\mathbb{R}^{n}$,
		there exists a sequence $\left\{ \eta_{m}^{g}\right\} $, such that
		\[
		\lim_{m\rightarrow\infty}\left\Vert f-\eta_{m}^{g}\right\Vert _{\mathcal{L}_{\infty}\left(\mathbb{K}\right)}=0.
		\]
	\end{lyxlist}
\end{thm}
Utilizing the techniques from \cite{Nestoridis:2007aa}, \cite{Bacharoglou:2010aa}
proved a similar set of results to Theorem \ref{thm nestoridis2007},
under the restriction that $f$ is a non-negative function with support
$\mathbb{R}$, using $g=\phi$ (i.e. $g$ has form (\ref{eq: gaussian pdf}),
where $n=1$) and taking $\left\{ h_{m}^{\phi}\right\} $ as the approximating
sequence, instead of $\left\{ \eta_{m}^{g}\right\} $. That is, the
following result is obtained.
\begin{thm}
	[Bacharoglou, 2010, Cor. 2.5] \label{thm Bacharoglou 2010}If $f:\mathbb{R}\rightarrow\mathbb{R}_{+}\cup\left\{ 0\right\} $,
	then the following statements are true.
	\begin{lyxlist}{00.00.0000}
		\item [{(a)}] For any pdf $f\in\mathcal{C}_{c}$, there exists a sequence
		$\left\{ h_{m}^{\phi}\right\} $ ($h_{m}^{\phi}\in\mathcal{M}_{m}^{\phi}$), such that
		\[
		\lim_{m\rightarrow\infty}\left\Vert f-h_{m}^{\phi}\right\Vert _{\mathcal{L}_{1}}+\left\Vert f-h_{m}^{\phi}\right\Vert _{\mathcal{L}_{\infty}}=0.
		\]
		\item [{(b)}] For any $f\in\mathcal{C}_{0}$, such that $\left\Vert f\right\Vert _{\mathcal{L}_{1}}\le1$,
		there exists a sequence $\left\{ h_{m}^{\phi}\right\} $ ($h_{m}^{\phi}\in\mathcal{M}_{m}^{\phi}$), such that
		\[
		\lim_{m\rightarrow\infty}\left\Vert f-h_{m}^{\phi}\right\Vert _{\mathcal{L}_{\infty}}=0.
		\]
		\item [{(c)}] For any $1<p<\infty$ and $f\in\mathcal{C}\cap\mathcal{L}_{p}$,
		such that $\left\Vert f\right\Vert _{\mathcal{L}_{1}}\le1$, there exists a sequence $\left\{ h_{m}^{\phi}\right\} $ ($h_{m}^{\phi}\in\mathcal{M}_{m}^{\phi}$),
		such that
		\[
		\lim_{m\rightarrow\infty}\left\Vert f-h_{m}^{\phi}\right\Vert _{\mathcal{L}_{p}}=0.
		\]
		\item [{(d)}] For any measurable $f$, there exists a sequence $\left\{ h_{m}^{\phi}\right\} $ ($h_{m}^{\phi}\in\mathcal{M}_{m}^{\phi}$),
		such that
		\[
		\lim_{m\rightarrow\infty}h_{m}^{\phi}=f\text{, almost everywhere.}
		\]
		\item [{(e)}] For any pdf $f\in\mathcal{C}$, there exists a sequence $\left\{ h_{m}^{\phi}\right\} $ ($h_{m}^{\phi}\in\mathcal{M}_{m}^{\phi}$),
		such that
		\[
		\lim_{m\rightarrow\infty}\left\Vert f-h_{m}^{\phi}\right\Vert _{\mathcal{L}_{1}}=0.
		\]
	\end{lyxlist}
\end{thm}
To the best of our knowledge, Theorem \ref{thm Bacharoglou 2010}
is the most complete characterization of the approximating capabilities
of the mixture of normal distributions. However, it is restrictive
in two ways. First, it does not permit characterization of approximation
via the class $\mathcal{M}_{m}^{g}$ for any $g$ except the normal
pdf $\phi$. Although $\phi$ is traditionally the most common choice
for $g$ in practice, the modern mixture model literature has seen
the use of many more exotic component pdfs, such as the student-\emph{t}
pdf and its skew and modified variants (see, e.g., \cite{Peel:2000aa},
\cite{Forbes:2013aa}, and \cite{Lee:2016aa}). Thus, its use
is somewhat limited in the modern context. Furthermore, modern applications
tend to call for $n>1$, further restricting the impact of the result
as a theoretical bulwark for finite mixture modeling in practice.
A remark in \cite{Bacharoglou:2010aa} states that the result can
generalized to the case where $g\in\mathcal{V}$ instead of $g=\phi$.
However, no suggestions were proposed, regarding the generalization
of Theorem \ref{thm Bacharoglou 2010} to the case of $n>1$.

In this article, we prove a novel set of results that largely generalize
Theorem \ref{thm Bacharoglou 2010}. Using techniques inspired by \cite{Donahue:1997aa}
and \cite{CheneyLight2000}, we are able to obtain a set of results
regarding the approximation capability of the class of $m\text{-component}$
mixture models $\mathcal{M}_{m}^{g}$, when $g\in\mathcal{C}_{0}$
or $g\in\mathcal{V}$, and for any $n\in\mathbb{N}$. By definition
of $\mathcal{V}$, the majority of our results extend beyond the proposed possible generalizations of Theorem \ref{thm Bacharoglou 2010}.

The article proceeds as follows. Our main theorem is stated and its seperate parts are proved in the Section \ref{mainResult}. Comments and discussion
are provided in Section \ref{Comments}. Necessary technical lemmas and results
are also included, for reference, in the Appendix.

\section{Main result} \label{mainResult}
The remainder of the article is devoted to proving the following theorem.
\begin{thm}
	[Main result] \label{thm main res}If we assume that $f$ and $g$
	are pdfs and that $g\in\mathcal{C}_{0}$, then the following statements
	are true.
	\begin{lyxlist}{00.00.0000}
		\item [{(a)}] For any $f\in\mathcal{C}_{0}$, there exists a sequence $\left\{ h_{m}^{g}\right\} $ ($h_{m}^{g}\in\mathcal{M}_{m}^{g}$),
		such that
		\[
		\lim_{m\rightarrow\infty}\left\Vert f-h_{m}^{g}\right\Vert _{\mathcal{L}_{\infty}}=0.
		\]
		\item [{(b)}] For any $f\in\mathcal{C}_{b}$ and compact $\mathbb{K}\subset\mathbb{R}^{n}$,
		there exists a sequence $\left\{ h_{m}^{g}\right\} $ ($h_{m}^{g}\in\mathcal{M}_{m}^{g}$), such that
		\[
		\lim_{m\rightarrow\infty}\left\Vert f-h_{m}^{g}\right\Vert _{\mathcal{L}_{\infty}\left(\mathbb{K}\right)}=0.
		\]
		\item [{(c)}] For any $1<p<\infty$ and $f\in\mathcal{L}_{p}$, there exists
		a sequence $\left\{ h_{m}^{g}\right\} $ ($h_{m}^{g}\in\mathcal{M}_{m}^{g}$), such that
		\[
		\lim_{m\rightarrow\infty}\left\Vert f-h_{m}^{g}\right\Vert _{\mathcal{L}_{p}}=0.
		\]
		\item [{(d)}] For any measurable $f$, there exists a sequence $\left\{ h_{m}^{g}\right\} $ ($h_{m}^{g}\in\mathcal{M}_{m}^{g}$),
		such that
		\[
		\lim_{m\rightarrow\infty}h_{m}^{g}=f\text{, almost everywhere.}
		\]
		\item [{(e)}] If $\nu$ is a $\sigma\text{-finite}$ Borel measure on $\mathbb{R}^{n}$,
		then for any $\nu\text{-measurable}$ $f$, there exists a sequence
		$\left\{ h_{m}^{g}\right\} $ ($h_{m}^{g}\in\mathcal{M}_{m}^{g}$), such that
		\[
		\lim_{m\rightarrow\infty}h_{m}^{g}=f,\] almost everywhere, with respect to $\nu$.
	\end{lyxlist}
	If we assume instead that $g\in\mathcal{V}$, then the following statement
	is also true.
	\begin{lyxlist}{00.00.0000}
		\item [{(f)}] For any $f\in\mathcal{C}$, there exists a sequence $\left\{ h_{m}^{g}\right\} $ ($h_{m}^{g}\in\mathcal{M}_{m}^{g}$),
		such that
		\[
		\lim_{m\rightarrow\infty}\left\Vert f-h_{m}^{g}\right\Vert _{\mathcal{L}_{1}}=0.
		\]
	\end{lyxlist}
\end{thm}

\subsection{Technical preliminaries}

Before we begin to prove the main theorem, we establish some technical results regarding our
class of component densities $\mathcal{C}_{0}$. Let $f,g\in\mathcal{L}_{1}$
and denote the convolution of $f$ and $g$ by $f\star g=g\star f$. Further, we denote the sequence of dilates of $g$ by  $\left\{ g_{k}:g_{k}\left(x\right)=k^{n}g\left(kx\right), k \in \mathbb{N}\right\}.$
The following result is an alternative to Lemma \ref{lem makarov 9.3.3}
and Corollary \ref{cor cheney ch20 th4}. Here, we replace a boundedness
assumption on the approximand, in the aforementioned theorem by a
vanishing at infinity assumption, instead.
\begin{lem}
	\label{lem approx ident}Let $g$ be a pdf and $f\in\mathcal{C}_{0}$,
	such that $\left\Vert f\right\Vert _{\mathcal{L}_{\infty}}>0$. Then,
	\[
	\lim_{k\rightarrow\infty}\left\Vert g_{k}\star f-f\right\Vert _{\mathcal{L}_{\infty}}=0\text{.}
	\]
\end{lem}
\begin{proof}
	It suffices to show that for any $\epsilon>0$, there exists a $k\left(\epsilon\right)\in\mathbb{N}$,
	such that $\left\Vert g_{k}\star f-f\right\Vert _{\mathcal{L}_{\infty}}<\epsilon$,
	for all $k\ge k\left(\epsilon\right)$. By Lemma \ref{lem C0 in Cb},
	$f\in\mathcal{C}_{b}$, and thus $\left\Vert f\right\Vert _{\mathcal{L}_{\infty}}<\infty$.
	By making the substitution $z=kx$, we obtain for each $k$
	\[
	\int g_{k}\left(x\right)\text{d}\lambda=\int k^{n}g\left(kx\right)\text{d}\lambda=\int g\left(z\right)\text{d}\lambda=1\text{.}
	\]
	By Corollary \ref{cor cheney ch20 th4}, we obtain
	$\lim_{k\rightarrow\infty}\int\mathbf{1}_{\left\{ x:\left\Vert x\right\Vert _{2}>\delta\right\} }g_{k}\text{d}\lambda=0$
	and thus we can choose a $k\left(\epsilon\right)$, such that
	\[
	\int\mathbf{1}_{\left\{ x:\left\Vert x\right\Vert _{2}>\delta\right\} }g_{k}\text{d}\lambda<\frac{\epsilon}{4\left\Vert f\right\Vert _{\mathcal{L}_{\infty}}}\text{.}
	\]
	Since $g$ is a pdf, we have
	\begin{align*}
	\left|\left(g_{k}\star f\right)(x)-f(x)\right| & =\left|\int g_{k}\left(y\right)\left[f\left(x-y\right)-f\left(x\right)\right]\text{d}\lambda\left(y\right)\right|\\
	& \le\int g_{k}\left(y\right)\left|f\left(x-y\right)-f\left(x\right)\right|\text{d}\lambda\left(y\right)\text{.}
	\end{align*}
	By uniform continuity, for any $\epsilon>0$, there exists a $\delta\left(\epsilon\right)>0$
	such that $\left|f\left(x-y\right)-f\left(x\right)\right|<\epsilon/2$,
	for any $x,y\in\mathbb{R}^{n}$, such that $\left\Vert y\right\Vert _{2}<\delta\left(\epsilon\right)$
	(Lemma \ref{lem C0 in Cb}). Thus, on the one hand, for any $\delta\left(\epsilon\right)$,
	we can pick a $k\left(\epsilon\right)$ such that
	\begin{align}
	&\nonumber\int\mathbf{1}_{\left\{ y:\left\Vert y\right\Vert _{2}>\delta\left(\epsilon\right)\right\} }g_{k}\left(y\right)\left|f\left(x-y\right)-f\left(x\right)\right|\text{d}\lambda\left(y\right)\\\nonumber & \le2\left\Vert f\right\Vert _{\mathcal{L}_{\infty}}\int\mathbf{1}_{\left\{ y:\left\Vert y\right\Vert _{2}>\delta\left(\epsilon\right)\right\} }g_{k}\text{d}\lambda \nonumber\\
	& \le2\left\Vert f\right\Vert _{\mathcal{L}_{\infty}}\times\frac{\epsilon}{4\left\Vert f\right\Vert _{\mathcal{L}_{\infty}}}=\frac{\epsilon}{2}\text{,}\label{eq: C0 approx identity 2} 
	\end{align}
	and on the other hand
	\begin{align}
	&\nonumber\int\mathbf{1}_{\left\{ y:\left\Vert y\right\Vert _{2}\le\delta\left(\epsilon\right)\right\} }g_{k}\left(y\right)\left|f\left(x-y\right)-f\left(x\right)\right|\text{d}\lambda\left(y\right)\\ & \le\frac{\epsilon}{2}\int\mathbf{1}_{\left\{ y:\left\Vert y\right\Vert _{2}\le\delta\left(\epsilon\right)\right\} }g_{k}\text{d}\lambda\nonumber\\
	& \le\frac{\epsilon}{2}\times1=\frac{\epsilon}{2}\text{.}\label{eq: C0 approx identity 3}
	\end{align}
	The proof is completed by summing (\ref{eq: C0 approx identity 2})
	and (\ref{eq: C0 approx identity 3}).
\end{proof}
\begin{lem}
	\label{lem urysohn implication}If $f\in\mathcal{C}_{0}$ is such
	that $f\ge0$, and $\epsilon>0$, then there exists a $h\in\mathcal{C}_{c}$,
	such that $0\le h\le f$, and
	\[
	\left\Vert f-h\right\Vert _{\mathcal{L}_{\infty}}<\epsilon
	\]
\end{lem}
\begin{proof}
	Since $f\in\mathcal{C}_{0}$, there exists a compact $\mathbb{K}\subset\mathbb{R}^{n}$
	such that $\left\Vert f\right\Vert _{\mathcal{L}_{\infty}\left(\mathbb{K}^{\complement}\right)}<\epsilon/2$.
	By Lemma \ref{lem urysohn}, there exists some $g\in\mathcal{C}_{c}$,
	such that $0\le g\le1$ and $\mathbf{1}_{\mathbb{K}}g=1$. Let $h=gf$,
	which implies that $h\ge0$ and $0\le h\le f$. Furthermore, notice
	that $\mathbf{1}_{\mathbb{K}}\left(f-h\right)=0$ and $\left\Vert h\right\Vert _{\mathcal{L}_{\infty}}\le\left\Vert f\right\Vert _{\mathcal{L}_{\infty}}$,
	by construction. The proof is completed by observing that
	\begin{align*}
	\left\Vert f-h\right\Vert _{\mathcal{L}_{\infty}}&=\left\Vert f-h\right\Vert _{\mathcal{L}_{\infty}\left(\mathbb{K}^{\complement}\right)}\\ &\le\left\Vert f\right\Vert _{\mathcal{L}_{\infty}\left(\mathbb{K}^{\complement}\right)}+\left\Vert h\right\Vert _{\mathcal{L}_{\infty}\left(\mathbb{K}^{\complement}\right)} \\&\le2\left\Vert f\right\Vert _{\mathcal{L}_{\infty}\left(\mathbb{K}^{\complement}\right)}<\epsilon\text{.}
	\end{align*}
\end{proof}
For any $\delta>0$, uniformly continuous function $f$, let
\[
w\left(f,\delta\right)=\sup_{\left\{ x,y\in\mathbb{R}^{n}:\left\Vert x-y\right\Vert_2 \le\delta\right\} }\left|f\left(x\right)-f\left(y\right)\right|
\]
denote the modulus of continuity of $f$. Furthermore, define the
diameter of a set $\mathbb{X}\subset\mathbb{R}^{n}$ by $\text{diam}\left(\mathbb{X}\right)=\sup_{x,y\in\mathbb{X}}\left\Vert x-y\right\Vert _{2}$
and denote an open ball, centered at $x\in\mathbb{R}^{n}$ with radius
$r>0$ by $\mathbb{B}\left(x,r\right)=\left\{ y\in\mathbb{R}^{n}:\left\Vert x-y\right\Vert _{2}<r\right\} $.

Notice that the class $\mathcal{M}_{m}^{g}$ can be parameterized
as
\begin{align*}
\mathcal{M}_{m}^{g}&=\Bigg\{ h:h\left(x\right)=\sum_{i=1}^{m}c_{i}k_{i}^{n}g\left(k_{i}x-z_{i}\right), \\&\text{\hspace{2cm}}z_{i}\in\mathbb{R}^{n}\text{, }k_{i}\in\mathbb{R}_{+}\text{, }c\in\mathbb{S}^{m-1}\text{, }i\in\left[m\right]\Bigg\} \text{,}
\end{align*}
where $k_{i}=1/\sigma_{i}$ and $z_{i}=\mu_{i}/\sigma_{i}$. The following
result is the primary mechanism that permits us to construct finite
mixture approximations for convolutions of form $g_{k}\star f$. The
argument motivated by the approaches taken in Theorem 1 in \cite[Ch. 24]{CheneyLight2000},
\cite[Lem. 3.1]{Nestoridis:2007aa}, and \cite[Thm. 3.1]{Nestoridis:2011aa}.
\begin{lem}
	\label{lem core lemma}Let $f\in\mathcal{C}$ and $g\in\mathcal{C}_{0}$
	be pdfs. Furthermore, let $\mathbb{K}\subset\mathbb{R}^{n}$ be compact
	and $h\in\mathcal{C}_{c}$, where $\mathbf{1}_{\mathbb{K}^{\complement}}h=0$
	and $0\le h\le f$. Then for any $k\in\mathbb{N}$, there exists a
	sequence $\left\{ h_{m}^{g}\right\} $, such that
	\[
	\lim_{m\rightarrow\infty}\left\Vert g_{k}\star h-h_{m}^{g}\right\Vert _{\mathcal{L}_{\infty}}=0.
	\]
\end{lem}
\begin{proof}
	It suffices to show that for any  $k\in\mathbb{N}$ and  $\epsilon>0$, there exists a sufficiently large enough $m(\epsilon) \in \mathbb{N}$ so that for all $m \geq m(\epsilon),h_{m}^{g}\in\mathcal{M}_{m}^{g}$
	such that
	\begin{equation}
	\left\Vert g_{k}\star h-h_{m}^{g}\right\Vert _{\mathcal{L}_{\infty}}<\epsilon.\label{eq: core lemma 1}
	\end{equation}
	
	For any $k\in\mathbb{N}$, we can write
	\begin{align*}
	\left(g_{k}\star h\right)(x) & =\int g_{k}\left(x-y\right)h\left(y\right)\text{d}\lambda\left(y\right)\\&=\int\mathbf{1}_{\left\{ y:y\in\mathbb{K}\right\} }g_k\left(x-y\right)h\left(y\right)\text{d}\lambda\left(y\right)\\
	& =\int\mathbf{1}_{\left\{ y:y\in\mathbb{K}\right\} }k^{n}g\left(kx-ky\right)h\left(y\right)\text{d}\lambda\left(y\right)\\&=\int\mathbf{1}_{\left\{ z:z\in k\mathbb{K}\right\} }g\left(kx-z\right)h\left(\frac{z}{k}\right)\text{d}\lambda\left(z\right)\text{.}
	\end{align*}
	Here, $k\mathbb{K}$ is continuous image of a compact set, and hence
	is compact (cf. \cite[Thm. 4.14]{Rudin:1976aa}). By Lemma \ref{lem covering},
	for any $\delta>0$, there exists $\kappa_{i}\in\mathbb{R}^{n}$ ($i\in\left[m-1\right]$,
	$m\in\mathbb{N}$), such that $k\mathbb{K}\subset\bigcup_{i=1}^{m-1}\mathbb{B}\left(\kappa_{i},\delta/2\right)$.
	Further, if $\mathbb{B}_{i}^{\delta}=k\mathbb{K}\cap\mathbb{B}\left(\kappa_{i},\delta/2\right)$,
	then we have $k\mathbb{K}=\bigcup_{i=1}^{m-1}\mathbb{B}_{i}^{\delta}$.
	We can obtain a disjoint covering of $k\mathbb{K}$ by taking $\mathbb{A}_{1}^{\delta}=\mathbb{B}_{1}$
	and $\mathbb{A}_{i}^{\delta}=\mathbb{B}_{i}^{\delta}\backslash\bigcup_{j=1}^{i-1}\mathbb{B}_{j}^{\delta}$
	($i\in\left[m-1\right]$) and noting that $k\mathbb{K}=\bigcup_{i=1}^{m-1}\mathbb{A}_{i}^{\delta}$,
	by construction (cf. \cite[Ch. 24]{CheneyLight2000}). Furthermore,
	each $\mathbb{A}_{i}^{\delta}$ is a Borel set and $\text{diam}\left(\mathbb{A}_{i}^{\delta}\right)\le\delta$.
	
	For convenience, let $\Pi_{m}^{\delta}=\left\{ \mathbb{A}_{i}^{\delta}:i\in\left[m-1\right]\right\} $
	denote the disjoint covering, or partition, of $k\mathbb{K}$. We
	seek to show that there exists an $m\in\mathbb{N}$ and $\Pi_{m}^{\delta}$,
	such that
	\[
	\left\Vert g_{k}\star h-\sum_{i=1}^{m}c_{i}k_{i}^{n}g\left(k_{i}x-z_{i}\right)\right\Vert _{\mathcal{L}_{\infty}}<\epsilon\text{,}
	\]
	where $k_{i}=k$,
	$$c_{i}=k^{-n}\int\mathbf{1}_{\left\{ z:z\in\mathbb{A}_{i}^{\delta}\right\} }h\left(z/k\right)\text{d}\lambda(z),$$
	and $z_i\in\mathbb{A}_{i}^{\delta}$, for $i\in\left[m-1\right]$.
	
	Further, $z_{m}\in\mathbb{A}_{m-1}^{\delta}$ and $c_{m}=1-\sum_{i=1}^{m-1}c_{i}$,
	with $k_{m}$ chosen as follows. By Lemma \ref{lem C0 in Cb}, $g\le C<\infty$
	for some positive $C$. Then, $\left\Vert c_{m}k_{m}^{n}g\left(k_{m}x-z_{m}\right)\right\Vert _{\mathcal{L}_{\infty}}\le c_{m}k_{m}^{n}C$.
	We may choose $k_{m}$ so that $k_{m}^{n}=\epsilon/\left(2c_{m}C\right)$,
	so that
	\[
	\left\Vert c_{m}k_{m}^{n}g\left(k_{m}x-z_{m}\right)\right\Vert _{\mathcal{L}_{\infty}}\le\frac{\epsilon}{2}\text{.}
	\]
	
	Since $0\le h\le f$, the sum of $c_{i}$ ($i\in\left[m-1\right]$)
	satisfies the inequality
	\begin{align*}
	\sum_{i=1}^{m-1}c_{i} & =k^{-n}\sum_{i=1}^{m-1}\int\mathbf{1}_{\left\{ z:z\in\mathbb{A}_{i}^{\delta}\right\} }h\left(\frac{z}{k}\right)\text{d}\lambda\\&=k^{-n}\int\mathbf{1}_{\left\{ z:z\in k\mathbb{K}\right\} }h\left(\frac{z}{k}\right)\text{d}\lambda\\
	& =\int\mathbf{1}_{\left\{ x:x\in\mathbb{K}\right\} }h\text{d}\lambda\le\int\mathbf{1}_{\left\{ x:x\in\mathbb{K}\right\} }f\text{d}\lambda\le\int f\text{d}\lambda=1\text{.}
	\end{align*}
	Thus, $0\le c_{m}\le1$, and our construction implies that $h_{m}^{g} \in \mathcal{M}_{m}^{g}\text{,}$ where
	\[
	h_{m}^{g}\left(x\right)=\sum_{i=1}^{m}c_{i}k_{i}^{n}g\left(k_{i}x-z_{i}\right) \forall x \in \mathbb{R}^n\text{.}
	\]
	
	We can bound the left-hand side of (\ref{eq: core lemma 1}) as follows:
	\small
	\begin{align}\label{eq: core lemma 2}
	& \left\Vert g_{k}\star h-h_{g}^{m}\right\Vert _{\mathcal{L}_{\infty}}\nonumber \\
	& \le\left\Vert \left(g_{k}\star h\right)\left(x\right)-\sum_{i=1}^{m-1}c_{i}k_{i}^{n}g\left(k_{i}x-z_{i}\right)\right\Vert _{\mathcal{L}_{\infty}} \nonumber\\& \text{\hspace{0.5cm}}+\left\Vert c_{m}k_{m}^{n}g\left(k_{m}x-z_{m}\right)\right\Vert _{\mathcal{L}_{\infty}}\nonumber \\
	& \le\left\Vert \left(g_{k}\star h\right)\left(x\right)-\sum_{i=1}^{m-1}c_{i}k_{i}^{n}g\left(k_{i}x-z_{i}\right)\right\Vert _{\mathcal{L}_{\infty}}+\frac{\epsilon}{2}\nonumber \\
	& =\Bigg\Vert \int\mathbf{1}_{\left\{ z:z\in k\mathbb{K}\right\} }g\left(kx-z\right)h\left(\frac{z}{k}\right)\text{d}\lambda\left(z\right)\nonumber\\& \text{\hspace{0.5cm}}-\sum_{i=1}^{m-1}\int\mathbf{1}_{\left\{ z:z\in\mathbb{A}_{i}^\delta\right\} }g\left(kx-z_{i}\right)h\left(\frac{z}{k}\right)\text{d}\lambda\left(z\right)\Bigg\Vert _{\mathcal{L}_{\infty}}+\frac{\epsilon}{2}\nonumber \\
	&\le \sum_{i=1}^{m-1} \int\mathbf{1}_{\left\{ z:z\in\mathbb{A}_{i}^\delta\right\} }\left\Vert g\left(kx-z\right)-g\left(kx-z_{i}\right) \right\Vert_{\mathcal{L}_{\infty}} h\left(\frac{z}{k}\right)\text{d}\lambda\left(z\right)  +\frac{\epsilon}{2}\text{.}
	\end{align}
	\normalsize
	Since
	\[
	\left\Vert kx-z-\left(kx-z_{i}\right)\right\Vert _{2}=\left\Vert z-z_{i}\right\Vert _{2}\le\text{diam}\left(\mathbb{A}_{i}^\delta\right)\le\delta\text{,}
	\]
	we have $\left|g\left(kx-z\right)-g\left(kx-z_{i}\right)\right|\le w\left(g,\delta\right)$,
	for each $i\in\left[m-1\right]$. Since $\lim_{\delta\rightarrow0}w\left(g,\delta\right)=0$
	(cf. \cite[Thm. 4.7.3]{Makarov:2013aa}), we may choose a $\delta\left(\epsilon\right)>0$
	so that $w\left(g,\delta\left(\epsilon\right)\right)<\epsilon/\left(2k^{n}\right)$.
	We may proceed from (\ref{eq: core lemma 2}) as follows:
	\begin{align}
	\left\Vert g_{k}\star h-h_{g}^{m}\right\Vert _{\mathcal{L}_{\infty}} & \le w\left(g,\delta\left(\epsilon\right)\right)\int\mathbf{1}_{\left\{ z:z\in k\mathbb{K}\right\} }h\left(\frac{z}{k}\right)\text{d}\lambda+\frac{\epsilon}{2}\nonumber \\
	& =w\left(g,\delta\left(\epsilon\right)\right)k^{n}\int h\text{d}\lambda+\frac{\epsilon}{2}\nonumber\\&\le w\left(g,\delta\left(\epsilon\right)\right)k^{n}+\frac{\epsilon}{2}\nonumber \\
	& <\frac{\epsilon}{2}+\frac{\epsilon}{2}=\epsilon\text{.}\label{eq: core lemma 3}
	\end{align}
	To conclude the proof, it suffices to choose an appropriate sequence of partitions
	$\Pi_{m}^{\delta\left(\epsilon\right)}, m \ge m(\epsilon)$, for some large but finite
	$m(\epsilon)$, so that (\ref{eq: core lemma 2}) and (\ref{eq: core lemma 3})
	hold, which is possible by Lemma \ref{lem covering}.
\end{proof}
For any $r\in\mathbb{N}$, let $\bar{\mathbb{B}}_{r}=\left\{ x\in\mathbb{R}^{n}:\left\Vert x\right\Vert _{2}\le r\right\} $
be a closed ball of radius $r$, centered at the origin.
\begin{lem}
	\label{lem Lesbegue dom con theorem}If $f\in\mathcal{L}_{1}$, such
	that $f\ge0$, then
	\[
	\lim_{r\rightarrow\infty}\left\Vert f-\mathbf{1}_{\bar{\mathbb{B}}_{r}}f\right\Vert _{\mathcal{L}_{1}}=0\text{.}
	\]
\end{lem}
\begin{proof}
	By construction, each element of the sequence $\left\{ \mathbf{1}_{\bar{\mathbb{B}}_{r}}f\right\} $
	($r\in\mathbb{N}$) is measurable, $0\le\mathbf{1}_{\bar{\mathbb{B}}_{r}}f\le f$,
	and
	\[
	\lim_{r\rightarrow\infty}\mathbf{1}_{\bar{\mathbb{B}}_{r}}f=f\text{,}
	\]
	point-wise. We obtain our conclusion via the Lesbegue dominated convergence
	theorem.
\end{proof}

\subsection{Proof of Theorem \ref{thm main res}(a)}

We now proceed to prove each of the parts of Theorem~\ref{thm main res}. To prove Theorem~\ref{thm main res}(a) it suffices to show that for every $\epsilon>0$, there exists a $h_{m}^{g}\in\mathcal{M}_{m}^{g}$,
such that $\left\Vert f-h_{m}^{g}\right\Vert _{\mathcal{L}_{\infty}}<\epsilon\text{.}$

Start by applying Lemma \ref{lem urysohn implication} to obtain $h\in\mathcal{C}_{c}$,
such that $0\le h\le f$ and $\left\Vert f-h\right\Vert _{\mathcal{L}_{\infty}}<\epsilon/2$.
Then, we have
\begin{align}
\left\Vert f-h_{m}^{g}\right\Vert _{\mathcal{L}_{\infty}} & \le\left\Vert f-h\right\Vert _{\mathcal{L}_{\infty}}+\left\Vert h-h_{m}^{g}\right\Vert _{\mathcal{L}_{\infty}}\nonumber \\
& <\frac{\epsilon}{2}+\left\Vert h-h_{m}^{g}\right\Vert _{\mathcal{L}_{\infty}}\text{.}\label{eq: proof (a) 1}
\end{align}

The goal is to find a $h_{m}^{g}$, such that $\left\Vert h-h_{m}^{g}\right\Vert _{\mathcal{L}_{\infty}}<\epsilon/2$.
Since $h\in\mathcal{C}_{c}$, we may find a compact $\mathbb{K}\subset\mathbb{R}^{n}$
such that $\left\Vert h\right\Vert _{\mathcal{L}_{\infty}\left(\mathbb{K}^{\complement}\right)}=0$.
Apply Lemma \ref{lem approx ident} to show the existence of a $k\left(\epsilon\right)$,
such that
\[
\left\Vert h-g_{k}\star h\right\Vert _{\mathcal{L}_{\infty}}<\frac{\epsilon}{4}\text{,}
\]
for all $k\ge k\left(\epsilon\right)$. With a fixed $k=k\left(\epsilon\right)$, apply Lemma \ref{lem core lemma} to show that there exists
a $h_{m}^{g}\in\mathcal{M}_{m}^{g}$, such that
\[
\left\Vert g_{k\left(\epsilon\right)}\star h-h_{m}^{g}\right\Vert _{\mathcal{L}_{\infty}}<\frac{\epsilon}{4}\text{.}
\]

By the triangle inequality, we have
\begin{align}
\left\Vert h-h_{m}^{g}\right\Vert _{\mathcal{L}_{\infty}}&\le\left\Vert h-g_{k\left(\epsilon\right)}\star h\right\Vert _{\mathcal{L}_{\infty}}+\left\Vert g_{k\left(\epsilon\right)}\star h-h_{m}^{g}\right\Vert _{\mathcal{L}_{\infty}}\nonumber\\&<\frac{\epsilon}{4}+\frac{\epsilon}{4}=\frac{\epsilon}{2}\text{.}\label{eq: proof (a) 2}
\end{align}
The proof is complete by substitution of (\ref{eq: proof (a) 2})
into (\ref{eq: proof (a) 1}).

\subsection{Proof of Theorem \ref{thm main res}(b)}

For any $\epsilon>0$ and compact $\mathbb{K}\subset\mathbb{R}^{n}$,
it suffices to show that there exists a sufficiently large enough $m(\epsilon) \in \mathbb{N}$ so that for all $m \geq m(\epsilon),h_{m}^{g}\in\mathcal{M}_{m}^{g},$ such that $\left\Vert f-h_{m}^{g}\right\Vert _{\mathcal{L}_{\infty}(\mathbb{K})}<\epsilon$.

By Lemma \ref{lem makarov 9.3.3}, we can find a $k\left(\epsilon,\mathbb{K}\right)\in\mathbb{N}$,
such that
\begin{equation}
\left\Vert f-g_{k}\star f\right\Vert _{\mathcal{L}_{\infty}\left(\mathbb{K}\right)}<\frac{\epsilon}{3}\text{,}\label{eq: lemma bound add 1}
\end{equation}
for every $k\ge k\left(\epsilon,\mathbb{K}\right)$. Since $g\in\mathcal{C}_{0}$,
$\left\Vert g\right\Vert _{\mathcal{L}_{\infty}}\le C<\infty$ for some positive $C$, by
Lemma \ref{lem C0 in Cb}. For any $k,r\in\mathbb{N}$, via Young's convolution inequality:

\begin{align}
\left\Vert g_{k}\star f-g_{k}\star\left(\mathbf{1}_{\bar{\mathbb{B}}_{r}}f\right)\right\Vert _{\mathcal{L}_{\infty}} \le k^{n}C\int\left(\mathbf{1}_{\bar{\mathbb{B}}_{r}^{\complement}}f\right)\text{d}\lambda=k^{n}C\left\Vert f-\mathbf{1}_{\bar{\mathbb{B}}_{r}}f\right\Vert _{\mathcal{L}_{1}}\text{.}\label{eq: lemma bounded 1}
\end{align}

For fixed $k$, we may choose $r\left(\epsilon,\mathbb{K}\right)\in\mathbb{N}$,
using Lemma \ref{lem Lesbegue dom con theorem}, so that $\left\Vert f-\mathbf{1}_{\bar{\mathbb{B}}_{r}}f\right\Vert _{\mathcal{L}_{1}}\le\epsilon/\left(3k^{n}C\right)$
and thus the final term of (\ref{eq: lemma bounded 1}) is bounded
from above by $\epsilon/3$ for all $r\ge r\left(\epsilon,\mathbb{K}\right)$.
Thus, for $k=k\left(\epsilon,\mathbb{K}\right)$
and, $r\ge r\left(\epsilon,\mathbb{K}\right)$
\begin{equation}
\left\Vert g_{k\left(\epsilon,\mathbb{K}\right)}\star f-g_{k\left(\epsilon,\mathbb{K}\right)}\star\left(\mathbf{1}_{\bar{\mathbb{B}}_{r\left(\epsilon,\mathbb{K}\right)}}f\right)\right\Vert _{\mathcal{L}_{\infty}}\le\frac{\epsilon}{3}\text{.}\label{eq: lemma bound add 2}
\end{equation}

Using Lemma \ref{lem core lemma}, with approximand $\mathbf{1}_{\bar{\mathbb{B}}_{r\left(\epsilon,\mathbb{K}\right)}}f$,
component density $g$, compact set $\bar{\mathbb{B}}_{r\left(\epsilon,\mathbb{K}\right)}$,
$h=\mathbf{1}_{\bar{\mathbb{B}}_{r\left(\epsilon,\mathbb{K}\right)}}f$,
and with $k=k\left(\epsilon,\mathbb{K}\right)$ fixed, we have the
existence of a density $h_{m}^{g}\in\mathcal{M}_{m}^{g}, m \geq m(\epsilon) \in \mathbb{N},$ such that
\begin{equation}
\left\Vert g_{k\left(\epsilon,\mathbb{K}\right)}\star\left(\mathbf{1}_{\bar{\mathbb{B}}_{r\left(\epsilon,\mathbb{K}\right)}}f\right)-h_{m}^{g}\right\Vert _{\mathcal{L}_{\infty}}\le\frac{\epsilon}{3}\text{.}\label{eq: lemma bound add 3}
\end{equation}
We obtain the desired result by combining (\ref{eq: lemma bound add 1}),
(\ref{eq: lemma bound add 2}), and (\ref{eq: lemma bound add 3}),
via the triangle inequality.

\subsection{Proof of Theorem \ref{thm main res}(c)}

The technique used to prove Theorem \ref{thm main res}(c) is different
to those used in the previous sections. Here, we use a result of \cite{Donahue:1997aa} that generalizes the classic Barron-Jones
Hilbert space approximation result (cf. \cite{Jones:1992aa} and
\cite{Barron:1993aa}) to Banach spaces.

To prove Theorem \ref{thm main res}(c), it suffices to show that for every $\epsilon>0$, there exists a sufficiently large enough $m(\epsilon) \in \mathbb{N}$ so that for all $m \geq m(\epsilon),h_{m}^{g}\in\mathcal{M}_{m}^{g}$
such that $\left\Vert f-h_m^g\right\Vert _{\mathcal{L}_{p}}<\epsilon$.
Begin by applying Corollary \ref{cor cheney ch20 th4} to obtain a
$k\left(\epsilon\right)$, such that
\begin{equation}
\left\Vert f-g_{k}\star f\right\Vert _{\mathcal{L}_{p}}<\frac{\epsilon}{2}\label{eq: lp triangle 1}
\end{equation}
for all $k\ge k\left(\epsilon\right)$.

For some pdf $g$ and fixed $k\in\mathbb{N}$, let us define the class
\[
\mathcal{G}_{g}^{k}=\left\{ h:h\left(x\right)=k^{n}g\left(kx-k\mu\right)\text{, }\mu\in\mathbb{R}^{n}\right\} \text{,}
\]
write the $m\text{-point}$ convex hull of $\mathcal{G}_{g}^{k}$
as
\begin{align*}
&\text{Conv}_{m}\left(\mathcal{G}_{g}^{k}\right)=\left\{ h:h=\sum_{i=1}^{m}c_{i}g_{i}\text{, }g_{i}\in\mathcal{G}_{g}^{k}\text{, }c\in\mathbb{S}^{m-1}\text{, }i\in\left[m\right]\right\} \text{,}
\end{align*}
and call $\text{Conv}_{\infty}\left(\mathcal{G}_{g}^{k}\right)=\text{Conv}\left(\mathcal{G}_{g}^{k}\right)$
the convex hull of $\mathcal{G}_{g}^{k}$. We further say that $\overline{\text{Conv}}\left(\mathcal{G}_{g}^{k}\right)$
is the closure of $\text{Conv}\left(\mathcal{G}_{g}^{k}\right)$.

Because $g$ is a pdf, $g\in\mathcal{C}_{0}\subset\mathcal{C}_{b}$,
and $\mathcal{C}_{b}\subset\mathcal{L}_{\infty}$, we observe that
$g\in\mathcal{L}_{1}\cap\mathcal{L}_{\infty}$. Thus, $g\in\mathcal{L}_{p}$,
for any $1<p<\infty$, by Lemma \ref{lem p and r in q}. Since $g$
is a pdf and $f\in\mathcal{L}_{p}$, we have the existence of $g_{k}\star f$
and the fact that $\left\Vert g_{k}\star f\right\Vert _{\mathcal{L}_{p}}$
is finite.

Furthermore, for any $\psi \in\mathcal{G}_{g}^{k}$, since $g\in\mathcal{L}_{p} $ and by definition of  $\mathcal{G}_{g}^{k}$,
we have $\left\Vert \psi\right\Vert _{\mathcal{L}_{p}}\le k^{n/p}\left\Vert g\right\Vert _{\mathcal{L}_{p}}.$ Thus, we have
\begin{equation}
\left\Vert \psi-g_{k}\star f\right\Vert _{\mathcal{L}_{p}}\le\left\Vert \psi\right\Vert _{\mathcal{L}_{p}}+\left\Vert g_{k}\star f\right\Vert _{\mathcal{L}_{p}}\le K\text{,}\label{eq: Main c 1}
\end{equation}
by choosing $K=k^{n/p}\left\Vert g\right\Vert _{\mathcal{L}_{p}}+\left\Vert g_{k}\star f\right\Vert _{\mathcal{L}_{p}}>0$.

Following \cite{van-de-Geer:2003aa}, we can write the closure of $\mathcal{G}_{g}^{k}$
as
\begin{align*}
&\overline{\text{Conv}}\left(G_{g}^{k}\right)=\left\{ h:h\left(x\right)=\int k^{n}g\left(kx-k\mu\right)f\left(\mu\right)\text{d}\lambda\left(\mu\right),f\text{ is a pdf}\right\} \text{,}
\end{align*}
and thus we immediately have $g_{k}\star f\in\overline{\text{Conv}}\left(G_{g}^{k}\right)$.
Combined with (\ref{eq: Main c 1}), we can apply Lemma \ref{lem Donahue}
to obtain the conclusion that there exists a function $h_{m}^{g}\in\text{Conv}_{m}\left(\mathcal{G}_{g}^{k\left(\epsilon\right)}\right)\subset\mathcal{M}_{m}^{g}$,
such that
\[
\left\Vert h_{m}^{g}-g_{k\left(\epsilon\right)}\star f\right\Vert _{\mathcal{L}_{p}}\le\frac{KC_{p}}{m^{1-1/\alpha}}\text{,}
\]
where $\alpha=\min\left\{ p,2\right\} $ and $C_{p}$ is a finite
constant. Since $p>1$, $m^{1-1/\alpha}$ is strictly increasing,
and hence we can choose an $m\left(\epsilon\right)\in\mathbb{N}$,
such that for all $m\ge m\left(\epsilon\right)$,
\begin{equation}
\left\Vert h_{m}^{g}-g_{k\left(\epsilon\right)}\star f\right\Vert _{\mathcal{L}_{p}}\le\frac{\epsilon}{2}\text{.}\label{eq: lp triangle 2}
\end{equation}
The proof is then completed by combining (\ref{eq: lp triangle 1}) and
(\ref{eq: lp triangle 2}) via the triangle inequality.

\subsection{Proof of Theorem \ref{thm main res}(d) and Theorem \ref{thm main res}(e)}

By Theorem \ref{thm main res}(a), there exists a sequence $\left\{ h_{m}^{g}\right\} $
that uniformly converges to $f$, as $m\rightarrow\infty$. Thus,
by Lemma \ref{lem Bartle},  $\left\{ h_{m}^{g}\right\} $ almost uniformly converges
to $f$ and also converges almost everywhere, to $f$, with respect
to any measure $\nu$. We prove Theorem \ref{thm main res}(d) by
setting $\nu=\lambda,$ and we prove Theorem \ref{thm main res}(e)
by not specifying $\nu$.

\subsection{Proof of Theorem \ref{thm main res}(f)}

It suffices to show that for any $\epsilon>0$, there exists a sufficiently large enough $m(\epsilon) \in \mathbb{N}$ so that for all $m \geq m(\epsilon),h_{m}^{g}\in\mathcal{M}_{m}^{g}$,
where $g\in\mathcal{V}$,
such that $\left\Vert f-h_{m}^{g}\right\Vert _{\mathcal{L}_{1}}<\epsilon$.
Begin by applying Lemma \ref{lem Lesbegue dom con theorem} in order
to find a $r\left(\epsilon\right)\in\mathbb{N}$, for any $\epsilon>0$,
such that for all $r\ge r\left(\epsilon\right)$,
\begin{equation}
\left\Vert f-\mathbf{1}_{\bar{\mathbb{B}}_{r}}f\right\Vert _{\mathcal{L}_{1}}\le\frac{\epsilon}{24}<\frac{\epsilon}{2}\text{,}\label{eq: epsilon/24}
\end{equation}
where $0\le\mathbf{1}_{\bar{\mathbb{B}}_{r}}f\le f$, and $\mathbf{1}_{\bar{\mathbb{B}}_{r}}f\in\mathcal{C}_{c}$
with compact support $\bar{\mathbb{B}}_{r}$.

Let $\mathbb{K}=\bar{\mathbb{B}}_{r}$ and apply the triangle inequality
to obtain
\begin{align*}
\left\Vert f-h_{m}^{g}\right\Vert _{\mathcal{L}_{1}}&\le\left\Vert f-\mathbf{1}_{\mathbb{K}}f\right\Vert _{\mathcal{L}_{1}}+\left\Vert \mathbf{1}_{\mathbb{K}}f-h_{m}^{g}\right\Vert _{\mathcal{L}_{1}}\\&\le\frac{\epsilon}{2}+\left\Vert \mathbf{1}_{\mathbb{K}}f-h_{m}^{g}\right\Vert _{\mathcal{L}_{1}}\text{.}
\end{align*}
Hence we need to show that there exists a function $h_{m}^{g}\in\mathcal{M}_{m}^{g}$,
such that
\[
\left\Vert \mathbf{1}_{\mathbb{K}}f-h_{m}^{g}\right\Vert _{\mathcal{L}_{1}}\le\frac{\epsilon}{2}\text{.}
\]

Since $g\in\mathcal{V}$ and $g_k(x) = k^n g(kx)$, by substitution, we have
\begin{equation}
g_{k}\left(x\right)\le\frac{\beta k^{-\theta}}{\left(k^{-1}+\left\Vert x\right\Vert _{2}\right)^{n+\theta}}\text{,}\label{eq: function in V}
\end{equation}
where $\beta,\theta>0$ are independent of $k$. By Lemma \ref{lem makarov 9.3.3}
and Corollary \ref{cor cheney ch20 th4}, we can obtain a $k_{1}\left(\epsilon\right)$,
such that for all $k\ge k_{1}\left(\epsilon\right)$,
\begin{equation}
\left\Vert \mathbf{1}_{\mathbb{K}}f-g_{k}\star\left(\mathbf{1}_{\mathbb{K}}f\right)\right\Vert _{\mathcal{L}_{1}}\le\frac{\epsilon}{4}\text{.}\label{eq: eps/4}
\end{equation}

Suppose that $\gamma>1$ and let $$\mathbb{K}_{k}=\left\{ x\in\mathbb{R}^{n}:\text{dist}\left(x,\mathbb{K}\right)\le k^{-\gamma}\right\}, $$ where
\[
\text{dist}\left(x,\mathbb{X}\right)=\inf\left\{ \left\Vert x-y\right\Vert _{2}:y\in\mathbb{X}\right\} \text{.}
\]
By construction, $\lambda\left(\mathbb{K}_{k}\right)=\lambda\left(\mathbb{K}\right)+O\left(k^{-\gamma}\right)$
and thus there exists a $k_{2}$ such that $\lambda\left(\mathbb{K}_{k}\right)\le\lambda\left(\mathbb{K}\right)+1$,
for any $k\ge k_{2}$.

For any $k>k_{2}$, we can show that
\begin{equation}
\left\Vert g_{k}\star\left(\mathbf{1}_{\mathbb{K}}f\right)-h_{m-1}^{g}\right\Vert _{\mathcal{L}_{1}\left(\mathbb{K}_{k}\right)}<\frac{\epsilon}{8}\text{.}\label{eq: eps/8}
\end{equation}
To do so, firstly, for any $x\in\mathbb{R}^{n}$,
\begin{align*}
g_{k}\star\left(\mathbf{1}_{\mathbb{K}}f\right) & =\int\mathbf{1}_{\mathbb{K}}g_{k}\left(x-y\right)f\left(y\right)\text{d}\lambda\left(y\right)\\
& =\int\mathbf{1}_{k\mathbb{K}}g\left(kx-z\right)f\left(\frac{z}{k}\right)\text{d}\lambda\left(z\right)\text{.}
\end{align*}

To obtain a Riemann sum approximation of $g_{k}\star\left(\mathbf{1}_{\mathbb{K}}f\right)$,
we use an argument analogous to that of Lemma \ref{lem core lemma}.
That is, we partition $k\mathbb{K}$ into $m-1$ disjoint Borel sets
$\Pi_{m}=\left\{ \mathbb{A}_{1},\dots,\mathbb{A}_{m-1}\right\} $,
and we approximate $g_{k}\star\left(\mathbf{1}_{\mathbb{K}}f\right)$
by a $h_{m-1}^{g}\in\mathcal{M}_{m-1}^{g}$, where for each $i\in\left[m-1\right]$,
$k_{i}=k$, $z_{i}\in\mathbb{A}_{i}$, and
\[
c_{i}=k^{-n}\int\mathbf{1}_{\mathbb{A}_{i}}f\left(\frac{z}{k}\right)\text{d}\lambda\left(z\right)\text{.}
\]
Define $k_{m}\in\mathbb{R}_{+}$, $z_{m}\in\mathbb{R}^{n}$, and $c_{m}=1-\sum_{i=1}^{m-1}c_{i}$,
where
\begin{equation}
c_{m}=\int f\text{d}\lambda-\int\mathbf{1}_{\mathbb{K}}f\text{d}\lambda=\left\Vert f-\mathbf{1}_{\mathbb{K}}f\right\Vert _{\mathcal{L}_{1}}\le\frac{\epsilon}{24}\label{eq: c bound}
\end{equation}
by (\ref{eq: epsilon/24}). Then, by a similar argument to Lemma \ref{lem core lemma},
$c_{i}\ge0$ for all $i\in\left[m\right]$ and $\sum_{i=1}^{m}c_{i}=1$.
Thus, we may define an element $h_{m}^{g}\in\mathcal{M}_{m}^{g}$
via the parameters above.

For sufficiently large $k\ge k_{2}$, we use Lemma \ref{lem core lemma}
to show that
\[
\left\Vert g_{k}\star\left(\mathbf{1}_{\mathbb{K}}f\right)-h_{m-1}^{g}\right\Vert _{\mathcal{L}_{\infty}\left(\mathbb{K}_{k}\right)}<\frac{\epsilon}{8\left(\lambda\left(\mathbb{K}\right)+1\right)}\text{,}
\]
which implies
\begin{align}
\left\Vert g_{k}\star\left(\mathbf{1}_{\mathbb{K}}f\right)-h_{m-1}^{g}\right\Vert _{\mathcal{L}_{1}\left(\mathbb{K}_{k}\right)} & <\int\mathbf{1}_{\mathbb{K}_{k}}\frac{\epsilon}{8\left(\lambda\left(\mathbb{K}\right)+1\right)}\text{d}\lambda\nonumber \\
& <\frac{\epsilon\lambda\left(\mathbb{K}_{k}\right)}{8\left(\lambda\left(\mathbb{K}\right)+1\right)}<\frac{\epsilon}{8}\text{,}\label{eq: eps 1/8 2}
\end{align}
and thus (\ref{eq: eps/8}) is proved. Using (\ref{eq: eps/8}), we
write
\begin{eqnarray*}
	&&\left\Vert g_{k}\star\left(\mathbf{1}_{\mathbb{K}}f\right)-h_{m}^{g}\right\Vert _{\mathcal{L}_{1}} \\& = & \left\Vert g_{k}\star\left(\mathbf{1}_{\mathbb{K}}f\right)-h_{m-1}^{g}-c_{m}k_{m}^{n}g\left(k_{m}x-z_{m}\right)\right\Vert _{\mathcal{L}_{1}}\\
	& \le & \left\Vert g_{k}\star\left(\mathbf{1}_{\mathbb{K}}f\right)-h_{m-1}^{g}\right\Vert _{\mathcal{L}_{1}\left(\mathbb{K}_{k}\right)}\\ &&+\left\Vert g_{k}\star\left(\mathbf{1}_{\mathbb{K}}f\right)-h_{m-1}^{g}\right\Vert _{\mathcal{L}_{1}\left(\mathbb{K}_{k}^{\complement}\right)}\\
	&  & +\left\Vert c_{m}k_{m}^{n}g\left(k_{m}x-z_{m}\right)\right\Vert _{\mathcal{L}_{1}}\\
	& \le & \frac{\epsilon}{8}+c_{m}+\left\Vert g_{k}\star\left(\mathbf{1}_{\mathbb{K}}f\right)\right\Vert _{\mathcal{L}_{1}\left(\mathbb{K}_{k}^{\complement}\right)}+\left\Vert h_{m-1}^{g}\right\Vert _{\mathcal{L}_{1}\left(\mathbb{K}_{k}^{\complement}\right)}\text{,}
\end{eqnarray*}
where $\left\Vert c_{m}k_{m}^{n}g\left(k_{m}x-z_{m}\right)\right\Vert _{\mathcal{L}_{1}}\le c_{m}$
since $k_{m}^{n}g\left(k_{m}x-z_{m}\right)$ is a pdf. The aim is
now to prove that
\[
\left\Vert g_{k}\star\left(\mathbf{1}_{\mathbb{K}}f\right)\right\Vert _{\mathcal{L}_{1}\left(\mathbb{K}_{k}^{\complement}\right)}<\frac{\epsilon}{24}\text{ and }\left\Vert h_{m-1}^{g}\right\Vert _{\mathcal{L}_{1}\left(\mathbb{K}_{k}^{\complement}\right)}<\frac{\epsilon}{24}\text{.}
\]

Using polar coordinates and \eqref{eq: function in V}, we have
\begin{align*}
&\int\mathbf{1}_{\left\{ x:\left\Vert x-y\right\Vert_2 >k^{-\gamma}\right\} }g_{k}\left(x-y\right)\text{d}\lambda\left(x\right) \\& \le\int\frac{\mathbf{1}_{\left\{ x:\left\Vert x-y\right\Vert_2 >k^{-\gamma}\right\} }\beta k^{-\theta}}{\left(k^{-1}+\left\Vert x-y\right\Vert _{2}\right)^{n+\theta}}\text{d}\lambda\left(x\right)\\
& =\beta A_{n}k^{-\theta}\int\frac{\mathbf{1}_{\left(k^{-\gamma},\infty\right)}r^{n-1}}{\left(k^{-1}+r\right)^{n+\theta}}\text{d}\lambda\left(r\right)\\
& \le\beta A_{n}k^{-\theta}\int\mathbf{1}_{\left(k^{-\gamma},\infty\right)}r^{-\theta-1}\text{d}\lambda\left(r\right)\\
& =\beta A_{n}k^{\theta\left(\gamma-1\right)}/\theta\text{,}
\end{align*}
where $A_{n}$ is the surface area of a unit sphere embedded in $\mathbb{R}^{n}$.
We then have

\begin{align*}
&\left\Vert g_{k}\star\left(\mathbf{1}_{\mathbb{K}}f\right)\right\Vert _{\mathcal{L}_{1}\left(\mathbb{K}_{k}^{\complement}\right)} \\& =\int\int\mathbf{1}_{\left\{ y\in\mathbb{K}\right\} }\mathbf{1}_{\left\{ x\in\mathbb{K}_{k}^{\complement}\right\} }f\left(y\right)g_{k}\left(x-y\right)\text{d}\lambda\left(x\right)\text{d}\lambda\left(y\right)\\
& \le\left\Vert \mathbf{1}_{\mathbb{K}}f\right\Vert _{\mathcal{L}_{\infty}}\int\int\mathbf{1}_{\left\{ y\in\mathbb{K}\right\} }\frac{\mathbf{1}_{\left\{ x:\left\Vert x-y\right\Vert_2 >k^{-\gamma}\right\} }\beta k^{-\theta}}{\left(k^{-1}+\left\Vert x-y\right\Vert _{2}\right)^{n+\theta}}\text{d}\lambda\left(x\right)\text{d}\lambda\left(y\right)\\
& \le\left\Vert \mathbf{1}_{\mathbb{K}}f\right\Vert _{\mathcal{L}_{\infty}}\lambda\left(\mathbb{K}\right)\beta A_{n}k^{\theta\left(\gamma-1\right)}/\theta\text{,}
\end{align*}
which implies that we can choose a $k_{3}\in\mathbb{N}$, such that
for all $k \ge k_{3}$,
\begin{equation}
\left\Vert g_{k}\star\left(\mathbf{1}_{\mathbb{K}}f\right)\right\Vert _{\mathcal{L}_{1}\left(\mathbb{K}_{k}^{\complement}\right)}<\frac{\epsilon}{24}\text{.}\label{eq: eps 1/24 2}
\end{equation}

Lastly, we write
\begin{align*}
&\left\Vert h_{m-1}^{g}\right\Vert _{\mathcal{L}_{1}\left(\mathbb{K}_{k}^{\complement}\right)} \\& =\int\mathbf{1}_{\mathbb{K}_{k}^{\complement}}\sum_{i=1}^{m-1}c_{i}k^{n}g\left(kx-z_{i}\right)\text{d}\lambda\\
& =\sum_{i=1}^{m-1}\int\mathbf{1}_{\mathbb{K}_{k}^{\complement}}\left[k^{-n}\int\mathbf{1}_{\mathbb{A}_{i}}f\left(\frac{z}{k}\right)\text{d}\lambda\left(z\right)\right]k^{n}g\left(kx-z_{i}\right)\text{d}\lambda\left(x\right)\\
& \le\left\Vert \mathbf{1}_{\mathbb{K}}f\right\Vert _{\mathcal{L}_{\infty}}\sum_{i=1}^{m-1}k^{-n}\lambda\left(\mathbb{A}_{i}\right)\int\mathbf{1}_{\mathbb{K}_{k}^{\complement}}g_{k}\left(x-\frac{z_{i}}{k}\right)\text{d}\lambda\\
& \le\left\Vert \mathbf{1}_{\mathbb{K}}f\right\Vert _{\mathcal{L}_{\infty}}\sum_{i=1}^{m-1}k^{-n}\lambda\left(\mathbb{A}_{i}\right)\frac{\beta A_{n}k^{\theta\left(\gamma-1\right)}}{\theta}\\
& \le\left\Vert \mathbf{1}_{\mathbb{K}}f\right\Vert _{\mathcal{L}_{\infty}}\lambda\left(\mathbb{K}\right)\beta A_{n}k^{\theta\left(\gamma-1\right)}/\theta\text{,}
\end{align*}
which implies that we can choose the same $k_{3}$ as above to obtain
the bound
\begin{equation}
\left\Vert h_{m-1}^{g}\right\Vert _{\mathcal{L}_{1}\left(\mathbb{K}_{k}^{\complement}\right)}<\frac{\epsilon}{24}\text{,}\label{eq: eps 1/24 3}
\end{equation}
for any $k \ge k_{3}$.

Thus, we obtain the bound $\left\Vert \mathbf{1}_{\mathbb{K}}f-h_{m}^{g}\right\Vert_{\mathcal{L}_1} <\epsilon/2$,
for all $k \ge \max\left\{ k_{1},k_{2},k_{3}\right\} $, by combining
(\ref{eq: eps/4}), (\ref{eq: eps/8}), (\ref{eq: c bound}), (\ref{eq: eps 1/8 2}),
(\ref{eq: eps 1/24 2}), and (\ref{eq: eps 1/24 3}), via the triangle
inequality. The result is proved by combing the bound above, with
(\ref{eq: epsilon/24}), for an appropriately large $r\left(\epsilon\right)\in\mathbb{N}$.

\section{Comments and discussion} \label{Comments}

\subsection{Relationship to Theorem \ref{thm zeevi}}

In the proof of Theorem \ref{thm zeevi}, the famous Hilbert space
approximation result of \cite{Jones:1992aa} and \cite{Barron:1993aa}
was used to bound the $\mathcal{L}_{2}$ norm between any approximand
$f\in\mathcal{L}_{2}$ and a convex combination of bounded functions
in $\mathcal{L}_{2}$. This approximation theorem is exactly the $p=2$
case of the more general theorem of \cite{Donahue:1997aa}, as presented
in Lemma \ref{lem Donahue}. Thus, one can view Theorem \ref{thm main res}(c)
as the $p\in\left(1,\infty\right)$ generalization of Theorem \ref{thm zeevi}.

\subsection{The class \texorpdfstring{$\mathcal{W}$}{Lg} is a proper subset of the class \texorpdfstring{$\mathcal{C}_{0}$}{Lg}}

Here, we comment on the nature of class $\mathcal{W}$, which was investigated by \cite{Bacharoglou:2010aa} and \cite{Nestoridis:2011aa}. We recall that \cite{Bacharoglou:2010aa} conjectured that Theorem
\ref{thm Bacharoglou 2010} generalizes from $g=\phi$ to $g\in\mathcal{V}$.
In Theorem \ref{thm main res}(a)--(e), we assume that $g\in\mathcal{C}_{0}$.
We can demonstrate that $g\in\mathcal{C}_{0}$ is a strictly weaker
condition than $g\in\mathcal{V}$ or $g\in\mathcal{W}$.

For example, consider the function in $g:\mathbb{R}\rightarrow\mathbb{R}$ such that $g\left(x\right) = 0$ if $x <0$ and 

\begin{align*}
g\left(x\right) &= \sum_{i=1}^{\infty}\frac{2^{2i}}{i}\Bigg[\left(x-i+1\right)^{2i}\mathbf{1}_{\left\{ i-1\le x<i-1/2\right\} }   +\left(x-i\right)^{2i}\mathbf{1}_{\left\{ i-1/2\le x<i\right\} }\Bigg] \text{ if }x\ge0\text{,}
\end{align*}
and note that
\[
\int\mathbf{1}_{\left(-1/2,1/2\right)}\frac{\left(2x\right)^{2i}}{i}\text{d}\lambda=\frac{1}{2i^{2}+i}<\frac{1}{i^{2}}\text{.}
\]
Since $\sum_{i=1}^{\infty}\left(1/i^{2}\right)=\pi^2/6$, $g\in\mathcal{L}_{1}$.
Furthermore, $g$ is continuous since all stationary points of $g$
are continuous. In $\mathbb{R}$, $g\in\mathcal{C}_{0}$ if
\[
\lim_{x\rightarrow\pm\infty}g\left(x\right)=0\text{.}
\]
For $x\le0$, we observe that $g=0$ and thus the left limit is satisfied.
On the right, for any $1/\epsilon>0$, we have $x\left(\epsilon\right)\ge\left\lceil \epsilon\right\rceil -1/2$,
so that $g\left(x\right)<1/\epsilon$, for all $x>x\left(\epsilon\right)$, where $\left\lceil \cdot\right\rceil$ is the ceiling operator.
Therefore, $g\in\mathcal{C}_{0}$.

Within each interval $i-1\le x<i$, we observe that $g$ is locally
maximized at $x=i-1/2$. The local maximum corresponding to each of
these points is $1/i$. Thus $g\notin\mathcal{W}$, since
\[
\sum_{i=1}^{\infty}\frac{1}{i}<\sum_{y\in\mathbb{Z}}\sup_{x\in\left[0,1\right]}\text{ }\left|g\left(x+y\right)\right|\text{,}
\]
where $\sum_{i=1}^{\infty}\left(1/i\right)=\infty$. Furthermore,
$g\notin\mathcal{V}$ since $\mathcal{V}\subset\mathcal{W}$.

\subsection{Convergence in measure}

Along with the conclusions of Theorem \ref{thm main res}(d) and (e),
Lemma \ref{lem Bartle} also implies convergence in measure. That
is, if $\nu$ is a $\sigma\text{-finite}$ Borel measure on $\mathbb{R}^{n}$,
then for any $\nu\text{-measurable}$ $f$, there exists a sequence
$\left\{ h_{m}^{g}\right\} $, such that for any $\epsilon>0$,
\[
\lim_{m\rightarrow\infty}\upsilon\left(\left\{ x\in\mathbb{R}^{n}:\left|f\left(x\right)-h_{m}^{g}\left(x\right)\right|\ge\epsilon\right\} \right)=0\text{.}
\]

%\appendices
\appendix

\section{Technical results}

Throughout the main text, we utilize a number of established technical
results. For the convenience of the reader, we append these results
within this Appendix. Sources from which we draw the unproved results
are provided at the end of the section.
\begin{lem}
	\label{lem makarov 9.3.3}Let $\left\{ g_{k}\right\} $ be a sequence of pdfs
	in $\mathcal{L}_{1}$ and for every $\delta>0$
	\[
	\lim_{k\rightarrow\infty}\int\mathbf{1}_{\left\{ x:\left\Vert x\right\Vert _{2}>\delta\right\} }g_{k}\text{d}\lambda=0\text{.}
	\]
	Then, for all $f\in\mathcal{L}_{p}$ and $1\le p<\infty$,
	\[
	\lim_{k\rightarrow\infty}\left\Vert g_{k}\star f-f\right\Vert _{\mathcal{L}_{p}}=0\text{.}
	\]
	Furthermore, for all $f\in\mathcal{C}_{b}$ and any compact $\mathbb{K}\subset\mathbb{R}^{n}$,
	\[
	\lim_{k\rightarrow\infty}\left\Vert g_{k}\star f-f\right\Vert _{\mathcal{L}_{\infty}\left(\mathbb{K}\right)}=0\text{.}
	\]
\end{lem}
The sequences $\left\{ g_{k}\right\} $ from Lemma \ref{lem makarov 9.3.3}
are often called approximate identities or approximations of the identity.
A simple construction of approximate identities is by taking dilations
$g_{k}\left(x\right)=k^{n}g\left(kx\right)$, which yields the following
corollary.
\begin{cor}
	\label{cor cheney ch20 th4}Let $g$ be a pdf. Then the sequence of
	dilations $\left\{ g_{k}:g_{k}\left(x\right)=k^{n}g\left(kx\right)\right\} $,
	satisfies the hypothesis of Lemma \ref{lem makarov 9.3.3} and hence
	permits its conclusion.
\end{cor}
\begin{lem}
	\label{lem C0 in Cb}The class $\mathcal{C}_{0}$ is a subset of $\mathcal{C}_{b}$.
	Furthermore, if $f\in\mathcal{C}_{0}$, then $f$ is uniformly continuous.
\end{lem}
\begin{lem}
	[Urysohn's Lemma] \label{lem urysohn}If $\mathbb{K}\subset\mathbb{R}^{n}$
	is compact, then there exists some $g\in\mathcal{C}_{c}$, such that
	$0\le g\le1$ and $\mathbf{1}_{\mathbb{K}}g=1$.
\end{lem}
\begin{lem}
	\label{lem covering}If $\mathbb{X}\subset\mathbb{R}^{n}$ is bounded,
	then for any $r>0$, $\mathbb{X}$ can be covered by $\bigcup_{i=1}^{m}\mathbb{B}\left(x_{i},r\right)$
	for some finite $m\in\mathbb{N}$, where $x_{i}\in\mathbb{R}^{n}$
	and $i\in\left[m\right]$.
\end{lem}
\begin{lem}
	\label{lem p and r in q}If $0<p<q<r\le\infty$, then $\mathcal{L}_{p}\cap\mathcal{L}_{r}\subset\mathcal{L}_{q}$.
\end{lem}
Let $\Gamma:\mathbb{R}\rightarrow\mathbb{R}$ be the usual gamma function,
defined as $\Gamma\left(z\right)=\int\mathbf{1}_{\left(0,\infty\right)}x^{z-1}\exp\left(-x\right)\text{d}\lambda$.
\begin{lem}
	\label{lem convolution bound}If $f\in\mathcal{L}_{p}$ and $g\in\mathcal{L}_{1}$,
	for $1\le p\le\infty$, then $f\star g$ exists and we have $\left\Vert f\star g\right\Vert _{\mathcal{L}_{p}}\le\left\Vert g\right\Vert _{\mathcal{L}_{1}}\left\Vert f\right\Vert _{\mathcal{L}_{p}}$.
\end{lem}
\begin{lem}
	\label{lem Donahue}Let $\mathcal{G}\subset\mathcal{L}_{p}$, for
	some $1\le p<\infty$, and let $f\in\overline{\text{Conv}}\left(\mathcal{G}\right)$.
	For any $K>0$, such that $\left\Vert f-\alpha\right\Vert _{\mathcal{L}_{p}}<K$,
	for all $\alpha\in\mathcal{G}$, there exists a $h_m\in\text{Conv}_{m}\left(\mathcal{G}\right)$,
	such that
	\[
	\left\Vert f-h_m\right\Vert _{\mathcal{L}_{p}}\le\frac{C_{p}K}{m^{1-1/\alpha}}\text{,}
	\]
	where $\alpha=\min\left\{ p,2\right\} $, and
	\[
	C_{p}=\begin{cases}
	1 & \text{if }1\le p\le2\text{,}\\
	\sqrt{2}\left[\sqrt{\pi}\Gamma\left(\frac{p+1}{2}\right)\right]^{1/p} & \text{if }p>2\text{.}
	\end{cases}
	\]
\end{lem}
\begin{lem}
	\label{lem Bartle}In any measure $\nu$, uniform convergence implies
	almost uniform convergence, and almost uniform convergence implies
	almost everywhere convergence and convergence in measure, with respect
	to $\nu$.
\end{lem}

\section{Sources of results}

Lemma \ref{lem makarov 9.3.3} is reported as Theorem 9.3.3 in \cite{Makarov:2013aa}
(see also Theorem 2 of \cite[Ch. 20]{CheneyLight2000}). The proof
of Corollary \ref{cor cheney ch20 th4} can be taken from that of
Theorem 4 of \cite[Ch. 20]{CheneyLight2000}. Lemma \ref{lem C0 in Cb}
appears in \cite{Conway:2012aa}, as Proposition 1.4.5. Lemma \ref{lem urysohn}
is taken from Corollary 1.2.9 of \cite{Conway:2012aa}. Lemma \ref{lem covering}
appears as Theorem 1.2.2 in \cite{Conway:2012aa}. Lemma \ref{lem p and r in q}
can be found in \cite[Prop. 6.10]{Folland:1999aa}. Lemma \ref{lem convolution bound}
can be found in \cite[Thm. 9.3.1]{Makarov:2013aa}. Lemma \ref{lem Donahue}
appears as Corollary 2.6 in \cite{Donahue:1997aa}. Lemma \ref{lem Bartle}
can be obtained from the definition of almost uniform convergence,
Lemma 7.10, and Theorem 7.11 of \cite{Bartle:1995aa}.

\section*{Acknowledgment}

HDN is personally funded by Australian Research Council (ARC) grant DE170101134. HDN and GJM are supported by ARC grant DP180101192. FC is supported by Agence Nationale de la Recherche (ANR) grant SMILES ANR-18-CE40-0014 and by R\'egion Normandie grant RIN AStERiCs.

\bibliographystyle{plain}
\bibliography{20190202_mixtures_denseness}

\end{document}